\documentclass[reqno]{amsart}

\usepackage{mathrsfs}
\usepackage{amsfonts}
\usepackage{amsmath}
\usepackage{amssymb}
\usepackage{mathtools}
\usepackage{float}
\usepackage{xcolor}
\usepackage[inline]{enumitem}
\usepackage{mathdots}
\usepackage{csquotes}
\usepackage{subcaption}
\usepackage{mathdots}
\usepackage{tikz}
\usepackage{tikz-cd}
\usetikzlibrary{arrows,automata}
\usetikzlibrary{decorations.markings}
\usetikzlibrary{automata, positioning, arrows}
\usetikzlibrary{math}
\usepackage{hyperref}
\usepackage{centernot}
\usepackage{varwidth}
\usepackage{bbm}

\newcommand{\pres}[3]{\textnormal{#1} \langle #2 \mid #3 \rangle}

\newcommand{\fR}{\mathfrak{R}}
\newcommand{\fA}{\mathfrak{A}}

\newcommand{\comment}[1]{\textcolor{red}{#1}}
\newcommand{\fAA}[1]{\mathfrak{A}^{(#1)}}

\newtheorem{innercustomgeneric}{\customgenericname}
\providecommand{\customgenericname}{}
\newcommand{\newcustomtheorem}[2]{%
  \newenvironment{#1}[1]
  {%
   \renewcommand\customgenericname{#2}%
   \renewcommand\theinnercustomgeneric{##1}%
   \innercustomgeneric
  }
  {\endinnercustomgeneric}
}
\newtheorem{theorem}{Theorem} % 1st argument is your name for it
\newtheorem*{theorem*}{Theorem} % 1st argument is your name for it
\newcustomtheorem{customthm}{Theorem}
\newcustomtheorem{customcor}{Corollary}
\newcustomtheorem{customprop}{Proposition}

\newtheorem{lemma}{Lemma}     % 2nd argument is what is printed

\newtheorem{proposition}{Proposition}
\numberwithin{lemma}{section}
\numberwithin{proposition}{section}

\newtheorem*{mainlemma*}{Main Lemma}

\theoremstyle{definition}

\newtheorem{question}{Question}
\newtheorem*{question*}{Question}
\newtheorem{example}{Example}
\numberwithin{example}{section}
\newtheorem{remark}{Remark}

\DeclareMathOperator{\BS}{BS}

\begin{document}

\title[]{On the Dehn functions of a class \\ of monadic one-relation monoids}

\author{Carl-Fredrik Nyberg-Brodda}
\address{Laboratoire d'Informatique Gaspard-Monge, Universit\'e Gustave Eiffel (Paris)}
%\curraddr{Department of Mathematics and Statistics,
%Case Western Reserve University, Cleveland, Ohio 43403}
\email{carl-fredrik.nyberg-brodda@univ-eiffel.fr}
\thanks{The author is currently working as Attach\'e Temporaire d'Enseignement et de Recherche at the Laboratoire d'Informatique Gaspard-Monge, Universit\'e Gustave Eiffel (Paris, France). }

%    General info
\subjclass[2020]{}

\date{\today}

%\dedicatory{}

\keywords{One-relation monoid; Dehn function; word problem}

\begin{abstract}
We give an infinite family of monoids $\Pi_N$ (for $N=2, 3, \dots$), each with a single defining relation of the form $bUa = a$, such that the Dehn function of $\Pi_N$ is at least exponential. More precisely, we prove that the Dehn function $\partial_N(n)$ of $\Pi_N$ satisfies $\partial_N(n) \succeq N^{n/4}$. This answers negatively a question posed by Cain \& Maltcev in 2013 on whether every monoid defined by a single relation of the form $bUa=a$ has quadratic Dehn function. Finally, by using the decidability of the rational subset membership problem in the metabelian Baumslag--Solitar groups $\BS(1,n)$ for all $n \geq 2$, proved recently by Cadilhac, Chistikov \& Zetzsche, we show that each $\Pi_N$ has decidable word problem.
\end{abstract}

\maketitle

\noindent The word problem for one-relation monoids is arguably the most important unsolved problem in semigroup theory, and has been open for over a century. Given two words $u, v$ over an alphabet $A$ (denoted $u, v \in A^\ast$), the monoid $M$ with the single defining relation $u=v$ is defined as the quotient of the free monoid $A^\ast$ by the least congruence containing the pair $(u, v)$. This monoid is denoted $\pres{Mon}{A}{u=v}$. The word problem for such a monoid -- a \textit{one-relation monoid} -- asks for a decision procedure which, given two words $w_1, w_2 \in A^\ast$, decides whether $(w_1, w_2)$ lies in the above congruence or not. Despite numerous efforts, this problem remains open. Thue \cite{Thue1914} gave a still inchoate study of this problem when the defining relation is of the form $u=1$, where $1$ is the empty word, and solved it in some particular cases. Adian in the 1960s made the first targeted effort for solving the problem, and solved the word problem in two important cases: first, whenever the relation is of the form $u=1$ (the \textit{special} case), and second, whenever the relation $u=v$ satisfies that the first letters of $u$ and $v$ differ, and the last letters of $u$ and $v$ differ (the \textit{cycle-free} or \textit{cancellative} case). Both results are proved via a reduction to Magnus' classical theorem on the decidability of the word problem in all one-relator groups \cite{Magnus1932}.

More detailed studies were later made by Adian \cite{Adian1976} and Adian \& Oganesian \cite{Adian1978, Adian1987}. These studies were focussed on the case of \textit{left cycle-free} monoids. In the case of one relation, this says that the relation $u=v$ has the first letters of $u$ and $v$ distinct, but the last letters may coincide. Adian \cite{Adian1976} introduced a pseudo-algorithm $\fA$ which, for any given left cycle-free monoid $M$, gives a procedure for studying the left divisibility problem in $M$. We detail this pseudo-algorithm in \S\ref{Sec:Adian-algo}. Using $\fA$, Adian \& Oganesian \cite{Adian1987} proved that the word problem for a given one-relation monoid $\pres{Mon}{A}{u=v}$ can be reduced to the word problem for a one-relation monoid of one of the forms:
\[
(1) \:\:\pres{Mon}{a,b}{bUa = aVa}, \qquad \text{or} \qquad (2) \:\: \pres{Mon}{a,b}{bUa = a},
\]
where $U, V$ are some words, i.e. the word problem for all one-relation monoids reduces to the word problem for all $2$-generated left cycle-free one-relation monoids. This reduction is very constructive, and easy to find in practice (see \cite[\S3, Example~3.14]{NybergBrodda2021} for details). The word problem remains open for left cycle-free monoids in general, even lifting the simplifying assumption of a single relation (see \cite[\S2.2]{NybergBrodda2021} for the general definition). Let $M_\ell$ be a left cycle-free monoid. Other than Adian's classical result that $M_\ell$ is left cancellative, some further results are known. For example, Valitskas proved that if $M_\ell$ is also right cancellative, then $M_\ell$ is group-embeddable.\footnote{This result was never published; Guba \cite[Theorem~4]{Guba1994} filled in the details in a later paper.} Furthermore, Guba \cite[Theorem~5.1]{Guba1997} proved that the set of principal right ideals of $M_\ell$ is a semi-lattice under intersection, and that any finitely generated right ideal of $M_\ell$ can be generated by two elements. These are strong structural properties not satisfied by all semigroups. 

A one-relation monoid of the form $\pres{Mon}{a,b}{bUa=a}$ will be called \textit{monadic}. These have seen a good deal of study. Oganesian, in particular, used $\fA$ to prove a number of remarkable statements in this setting. For example, Oganesian \cite{Oganesian1984} proved that the isomorphism problem for one-relation monoids can be reduced to the word problem for all monadic one-relation monoids, and further (see \cite{Oganesian1982}) proved that the word problem for all monadic one-relation monoids reduces to the left divisibility problem for all cycle-free monoids.\footnote{Sarkisian claimed a solution to the latter problem, but a gap was found in the 1990s, and it remains open. See \cite[\S4.4]{NybergBrodda2021} for further details.} More specifically, Oganesian proved the remarkable fact that if $M$ is a monadic one-relation monoid with defining relation $bUa = a$, then the monoid $S(M)$ generated by all suffixes of $bUa$ is cycle-free, and furthermore the word problem for $M$ reduces to the left divisibility problem in $S(M)$. Guba \cite{Guba1997} extended this latter result to prove that $S(M)$ embeds in a one-relator group $\overline{G}(M)$ as the submonoid generated by all suffixes of the defining relation of $\overline{G}(M)$. As the left divisibility problem in $S(M)$ reduces to the membership problem for $S(M)$ in $\overline{G}(M)$, this shows that the word problem for all monadic one-relation monoids reduces to the submonoid membership problem for all one-relator groups.

The starting point for the present article, which is focussed on a class of monadic one-relation monoids, comes in the form of two pre-prints from 2013, by Cain \& Maltcev \cite{Cain2013, Cain2013b}. These pre-prints attempt to construct \textit{finite complete rewriting systems} for monadic left-cycle free monoids $M$ in which the defining relation $bUa = a$ satisfies the condition that $bUa$ has a short ``relative length'', i.e. the shortest way, in terms of number of alternations, to write $bUa$ as an alternating product of $b$'s and $a$'s. Specifically, in \cite{Cain2013} it is shown that if $bUa = b^\alpha a^\beta b^\gamma a^\delta$ for some $\alpha, \beta, \gamma, \delta \geq 1$ (i.e. $bUa$ has relative length $4$), then $M$ admits a finite complete rewriting system. In \cite{Cain2013b}, the same claim was made for the case of $bUa$ having relative length $6$; however, as we shall see, the monoid $\Pi_2$ of this article demonstrates the existence of gap in the proof (but not, a priori, a counterexample to the statement), see \S\ref{Sec:Some-remarks}. 

Their approach was in line with answering the (still open) question of whether every one-relation monoid admits a finite complete rewriting system. Of course, a positive answer to this question would resolve the word problem, and so would represent an incredible leap forward from our current understanding; however, a negative answer may be obtained even without involving the word problem. Based on historical results, such a negative answer may be obtained via considering appropriate types of homological finiteness properties, see e.g. \cite{Kobayashi2000}. For recent progress on this subject, see the result by Gray \& Steinberg \cite{Gray2022} that every one-relation monoid satisfies the homological finiteness property $\operatorname{FP}_\infty$, a condition satisfied by all monoids admitting finite complete rewriting systems. 

Cain \& Maltcev noticed, based on their rewriting systems constructed, that the Dehn function of all monadic one-relation monoids (1) in which $bUa$ has relative length $4$ is at most quadratic. This prompted the authors to ask the following question, see \cite[Open Problem 5.1(1)]{Cain2013}:

\begin{question*}[Cain \& Maltcev, 2013]
Does every monoid $\pres{Mon}{a,b}{bUa=a}$ have at most quadratic Dehn function?
\end{question*}

Here, the Dehn function is a measure of the complexity of the ``na\"ive'' method of attempting to solve the word problem in a given monoid, see \S\ref{Sec:Adian-algo} for details. In this article, we provide a negative answer to their question. More specifically, we will show the following main theorem:

\begin{theorem}\label{Thm:Main_thm}
For every $N \geq 2$, the Dehn function $\partial_N$ of the one-relation monoid
\[
\Pi_N = \pres{Mon}{a,b}{baa(ba)^N = a}
\]
satisfies $\partial_N(n) \succeq N^{n/4}$, i.e. the Dehn function of $\Pi_N$ is at least exponential. 
\end{theorem}

Hence, Theorem~\ref{Thm:Main_thm} answers negatively Cain \& Maltcev's question. The proof of Theorem~\ref{Thm:Main_thm} is enacted via the proof of a Main Lemma, which shows that for every $k \geq 1$ and every $N \geq 2$ we have
\[
ab^{2k}a^{2k}a = b^{2k-1}a^{2k}baa \quad \text{in $\Pi_N$},
\]
but that the shortest sequence of elementary transformations of $\Pi_N$ verifying this grows exponentially in $k$ (with base $N$). 

In spite of the above result, in \S\ref{Sec:Some-remarks} we will also prove (Theorem~\ref{Thm:Pi_0-has-dec-wp}) that the word problem is decidable in all $\Pi_N$, by reducing it to the decidability of the rational subset membership problem in solvable Baumslag--Solitar groups. Finally, we will mention some links with other results, including residual finiteness, positive one-relator groups, finite complete rewriting systems, automaticity, and one-relation inverse monoids.

\section{Transformations in cycle-free monoids}\label{Sec:Adian-algo}

\noindent We assume the reader is familiar with the basics of combinatorial group and semigroup theory, in particular the theory of presentations. The reader may consult \cite{Adian1966, Magnus1966, Ruskuc1995} if they are not. We will denote monoid (resp. group) presentations as $\pres{Mon}{A}{R}$ resp. $\pres{Gp}{A}{R}$, where $A$ is the generating set, and $R$ is the set of defining relations resp. relators. The free monoid on an alphabet $A$ is denoted $A^\ast$. For two words $u, v \in A^\ast$ with a maximal shared prefix $p$, i.e. $u \equiv pu'$ and $v \equiv pv'$ ($p$ may be empty) and the first letters of $u'$ and $v'$ differ (or both are empty), we let $(u, v)_\text{p-red}$ denote the pair $(u', v')$, i.e. the pair obtained by removing the maximal shared prefix of $u$ and $v$. We say that the pair $(u, v)$ is \textit{(left-)reduced} if $(u, v)_\text{p-red} = (u, v)$, i.e. $u$ and $v$ have no shared non-empty prefix. We omit many of the details of Adian's theory of cycle-free monoids, and refer the reader to \cite{NybergBrodda2021}. 

\subsection{Dehn functions}

For a finitely presented monoid $M = \pres{Mon}{A}{R}$, each equivalence class of a word $w \in A^\ast$ can be regarded as a metric space. Concretely, if $u, v \in A^\ast$ are such that $u = v$ in $M$, then we can define a metric $\partial_M(u, v)$ as the shortest sequence of elementary transformations of $M$ transforming $u$ into $v$, i.e. the shortest sequence of applications of relations connecting $u$ and $v$. Overloading notation, this metric gives rise to the monotone non-decreasing \textit{Dehn function} $\partial_M \colon \mathbb{N} \to \mathbb{N}$ (also called the \textit{least isoperimetric function}) for $M$ is defined as:
\[
\partial_M(n) = \max \{ \partial_M(u, v) \colon u, v \in A^\ast, u =_M v \text{ and } |u| + |v| \leq n \}.
\]
Here $|v|$ denotes the length of $v \in A^\ast$ (as a word over the alphabet $A$). In other words, $\partial_M(n)$ measures the complexity of the ``na\"ive'' approach to solving the word problem in a monoid via enumerating equalities by using the defining relations. It is easy to see that the word problem for $M$ is decidable if and only if $\partial_M(n)$ is a recursive function (with respect to any generating set $A$). For two functions $\varphi, \psi \colon \mathbb{N} \to \mathbb{N}$, we write $\varphi \preceq \psi$ if there is a constant $c \in \mathbb{N}$ such that $\varphi(n) \leq c\psi(cn) + cn$ for all $n \in \mathbb{N}$, and we write $\varphi \sim \psi$ if $\varphi \preceq \psi$ and $\psi \preceq \varphi$. Then $\sim$ is an equivalence relation on the set of all functions $\varphi \colon \mathbb{N} \to \mathbb{N}$. The equivalence class -- i.e. the asymptotic growth rate -- of $\partial_M$ under $\sim$ can easily be shown to not depend on the particular finite generating set $A$ chosen (which justifies our somewhat abusive suppression of any mention of the generating set above). Thus, we may speak of a monoid having a linear, quadratic, polynomial, exponential, etc., Dehn function, where e.g. $M$ having a quadratic Dehn function means that $\partial_M(n) \sim n^2$. This is true, for example, for the monoid $\mathbb{N} \times \mathbb{N} \cong \pres{Mon}{a,b}{ab=ba}$.

\subsection{Algorithm $\fA$}

In 1976, Adian \cite{Adian1976} described a pseudo-algorithm for solving the left divisibility problem in a given left cycle-free monoid (and hence also the word problem). This pseudo-algorithm is called $\fA$, and takes as input the defining relations of a left cycle-free presentation for a monoid $M$, a word $w$, and a letter $a$, and outputs either ``yes'' or ``no'', depending on whether $w$ is left divisible by $a$ in $M$. It is termed a pseudo-algorithm rather than an algorithm because it does not always terminate, and it has no built-in mechanism for detecting non-termination, even for a fixed set of defining relations and letter. In the sequel, we will be somewhat sloppy and write $\fA$ as taking only a single word as input (with a minor modification below to allow for two words), and let context make the defining relations and the letter clear.

We will only describe $\fA$ for the case of a left cycle-free one-relation monoid $M = \pres{Mon}{a,b}{bP = aQ}$. For a full description, we refer the reader to \cite[\S4.2]{NybergBrodda2021}. We will make use of \textit{prefix decompositions} of a word (with respect to the above presentation). Let $ A= \{ a, b\}$. For a word $w \in A^\ast$, we find the \textit{prefix decomposition} $\fR(w)$ of $w$ as follows. From left to right, factorise $w$ into a product of maximal prefixes of either $bUa$ or $a$ (as $M$ is left cycle-free, this is well-defined). If any of these maximal prefixes is the entire defining word $bP$ or $aQ$, then we stop, and call this the \textit{head} of the $\fR(w)$. The remaining suffix of $w$ is called the \textit{tail} of $\fR(w)$. If $\fR(w)$ has no head, then it is called a \textit{headless} decomposition. We write 
\[
\fR(w) \equiv w_1 \mid w_2 \mid \cdots \mid w_k \: \fbox{$H$} \: w'
\]
to denote that the $w_i$ are maximal proper prefixes of $bP$ or $aQ$, and that $H \equiv bP$ or $H \equiv aQ$ is the head of the prefix decomposition of $w$; the word $w'$ is the tail. In the case of a headless decomposition, we write this simply
\[
\fR(w) \equiv w_1 \mid w_2 \mid \cdots \mid w_k.
\]

\begin{example}\label{Ex:prefix-decomps}
Let $M_0 = \pres{Mon}{a,b}{b^2a^2 = a}$. We give an example of two prefix decompositions with a head, and one that is headless.
\begin{align}
\fR(bbbbabbaabbab) &\equiv bb \mid bba \: \fbox{$bbaa$} \: bbab.\label{Ex1-decomp} \\
\fR(bbabbababab) &\equiv bba \mid bba \mid b \: \fbox{$a$} \: bab. \label{Ex2-decomp}\\
\fR(bbabbabb) &\equiv bba \mid bba \mid bb \label{Ex3-decomp}.
\end{align}
We shall see below that results due to Adian allow us to conclude, from the latter of these two decompositions, that $bbabbabb$ is \textit{not} left divisible by $a$ in $M_0$.
\end{example}

We shall make one notational simplification. If the prefix decomposition of a word $w$ begins with $p>1$ copies of a prefix $u$, i.e. if $\fR(w)$ is of the form
\begin{equation}\label{Eq:u**p-simplication}
\fR(w) = \underbrace{u \mid u \mid \cdots \mid u}_{p \text{ times}} \mid v \cdots 
\end{equation}
and none of these occurrences of $u$ are the head (if any) of the decomposition, then we will simply write this as 
\[
\fR(w) = u^p \mid v \cdots.
\]
This is not very abusive, as we know by the prefix decomposition \eqref{Eq:u**p-simplication} that $u^p$ is not itself a prefix of any defining relation. 

Adian's algorithm $\fA$, with input a word $w$ and a letter $x \in A$, is now described as following. First, if $w$ begins with $x$, the procedure halts, and outputs \textsc{yes}. Otherwise, compute the prefix decomposition $\fR(w)$. If $\fR(w)$ is headless, then halt, and output \textsc{no}. If $\fR(w)$ has a head, then replace the head by the other side of the defining relation. This results in a word $w'$. We say $\fAA{1}_x(w) = w'$ or simply $\fAA{1}(w) = w'$. We then iterate this procedure, finding $\fAA{2}(w), \fAA{3}(w), \dots$. This process does not always terminate.

\begin{example}
We continue Example~\ref{Ex:prefix-decomps}. Fixing the letter $x \equiv a$, and suppressing it in writing $\fA$ below, we have, using the decompositions \eqref{Ex1-decomp} and \eqref{Ex2-decomp}, that
\begin{align*}
\fAA{1}(bbbbabbaabbab) &= bbbba \cdot a \cdot bbab, \\
\fAA{1}(bbabbababab) &= bbabbab \cdot bbaa \cdot bab,
\end{align*}
and as $bbabbabb$ has a headless prefix decomposition, $\fA$ terminates immediately on input $bbabbabb$. We leave the reader to verify that 
\[
\fAA{2}(bbbbabbaabbab) = bbabbab,
\]
which has a headless prefix decomposition. Furthermore, it is not hard to see that $\fA$ does not terminate on input $bbabbabababab$ as above. Indeed, we even have 
\begin{align*}
\fAA{1}(ba) = b bbaa, \quad \fAA{2}(ba) = bbbbbaaa, \quad, \dots, \quad \fAA{i}(ba) = b^{2i+1}a^{i+1}, \dots 
\end{align*}
and so $\fA$ does not even terminate on input $ba$.
\end{example}

In general, as our alphabet consists of only two letters $a$ and $b$, when writing $\fA(w)$, we will generally implicitly assume that the letter $x$ is simply the letter which $w$ does not begin with (as otherwise $\fA$ would terminate immediately). 

\begin{theorem*}[Adian, 1976]
Let $M = \pres{Mon}{A}{R}$ be a left-cycle free monoid. Suppose $\fA$ is applied to a word $u \in A^\ast$. Then $u$ is left divisible by $x \in A$ if and only if $\fA$ terminates on input $u$ and $x$ after a finite number of steps in a word beginning with $x$. That is, $u$ is left divisible by $x$ if and only if $\fA$ outputs \textsc{yes} on input $u$ and $x$.
\end{theorem*}

In view of the remarkable result by Guba \cite[Theorem~4.1]{Guba1997} that the word problem is equivalent to the left divisibility problem in the case of monadic left cycle-free one-relation monoids, it follows that the halting problem for $\fA$ for the monoid is in these cases equivalent to the word problem for the monoid.

We shall make a notational modification for $\fA$, and use this to denote another, very closely related, algorithm. This will operate on pairs of words $u, v \in A^\ast$ with no common non-empty prefix, as follows. First, if either of $u$ or $v$ are empty, then $\fA$ halts. If both $u$ and $v$ are non-empty, then $\fA$ will perform a single step of Adian's (usual) algorithm $\fA$ on $u$. Suppose $\fAA{1}(u) = u'$. Then we set $\fAA{1}(u, v) = (u', v)_\text{p-red}$. As this results in a pair of words $(u'', v')$ with no common non-empty prefix, we may iterate $\fA$ on $(u'', v')$, obtaining $\fAA{2}(u, v)$, etc. When we write $\fAA{k}(u, v) = (\varepsilon, \varepsilon)$, we mean that $\fA$ takes \textit{exactly} $k$ steps to reach $(\varepsilon, \varepsilon)$ when applied to the pair $(u, v)$. 

The following is one of the key properties of $\fA$, and is a direct and easy reformulation of the proof of \cite[Theorem~1]{Adian1976}.

\begin{proposition}[Adian, 1976]\label{Prop:Adian-prop}
Let $M = \pres{Mon}{A}{R}$ be a left cycle-free monoid, and $u, v \in A^\ast$. Then $u =_M v$ if and only if there is some (necessarily unique) $k \in \mathbb{N}$ such that $\fAA{k}(u, v) = (\varepsilon, \varepsilon)$. Furthermore, if such $k$ exists, then $\partial_M(u, v) = k$. 
\end{proposition}

In other words, stated geometrically (and for the reader familiar with the details of semigroup diagrams, cf. e.g. \cite{Guba1997}), for equal words $u$ and $v$, $\fA$ can be used to produce a $(u,v)$-diagram with the least number of cells over all $(u, v)$-diagrams. Furthermore, this is the only $(u, v)$-diagram with this number of cells.

\section{Equalities in $\Pi_0$}

\subsection{Basic lemmas in $\Pi_N$}

As mentioned in the introduction, the monoids
\[
\Pi_N = \pres{Mon}{a,b}{baa(ba)^N = a},
\]
where $N  \geq 2$, hold center stage. In the sequel, we will so frequently use the abbreviation
\[
X \equiv ba
\]
that it deserves its own line. Note that this makes the defining relation of $\Pi_N$ as $XaX^N = a$. We state two lemmas on word transformations via $\fA$ as applied to $\Pi_N$. In the sequel, whenever we write ``word'' we mean ``word over $\{ a, b\}$''; whenever we write $\fA$, we mean $\fA$ as applied to $\Pi_N$; when we write ``$u$ is left divisible by $v$'' we mean ``$u$ is left divisible by $v$ in $\Pi_N$'', etc.

\begin{lemma}\label{Lem:left-comp-lemma}
Let $P, Q$ be arbitrary words. Suppose $\fAA{k}(P) = W$ for some word $W$ and $k \in \mathbb{N}$. If either 
\begin{enumerate}
\item $P$ begins with $a$, and is left divisible by $b$; or
\item $P$ begins with $b$, and is left divisible by $a$.
\end{enumerate}
Then we have $\fAA{k}(PQ) = WQ$.
\end{lemma}
\iffalse
\begin{proof}
All decompositions of any successor of $P$ has a head
\end{proof}
\fi

The proof is easy by using Adian's Theorem on $\fA$ (as above) to conclude that for all $k_0 < k$ the word $\fAA{k_0}(P)$ has a head, so the head of the prefix decomposition of $\fAA{k_0}(PQ)$ is inside this distinguished subword $P$. 

\begin{lemma}\label{Lem:double-comp-lemma}
Let $Y \equiv bY'$ be a word beginning with $b$. Let $\ell \geq 0$ be such that for all $0 \leq k \leq \ell$, we have that $\fAA{k}(Y)$ also begins with $b$. Then $\fAA{\ell}(b^s Y) = b^s \fAA{\ell}(Y)$.
\end{lemma}

The proof of this lemma is immediate by strong induction on $\ell$.

\subsection{The Main Lemma}

We now turn to the proof of Theorem~\ref{Thm:Main_thm}, which will pass via a Main Lemma. For every $k \geq 0$, we define the following pairs of words:
\[
U_k = ab^{2k}a^{2k}a, \qquad V_k = b^{2k-1}a^{2k}baa.
\]
The main equalities we shall prove is the following:
\begin{equation}\label{Eq:Main-equality}
U_k = V_k \quad \text{in $\Pi_N$ for all $k \geq 1$ and all $N \geq 2$}.
\end{equation}
We remark immediately that the words $U_k$ and $V_k$ do \textit{not} depend on $N$. That is, these words are equal in all $\Pi_N$. Indeed, we have
\[ 
|U_k| = |V_k| = 4k+2.
\]
We shall prove that although the equality $U_k = V_k$ holds in $\Pi_N$, this can only be proved using sequences of elementary transformations whose length are exponential, with exponent $k$ and base $N^2$. More precisely, we prove the following:

\begin{mainlemma*}
For every $k \geq 1$ and $N \geq 2$, we have $U_k = V_k$ in $\Pi_N$. Furthermore, the shortest chain of elementary transformations of $\Pi_0$ verifying this equality has length $\sigma_N(k)$, where
\[
\sigma_N(k) := \frac{2N}{N^2-1}(N^{2k} - 1) + 4k -2.
\]
In other words, we have $\partial_N(U_k, V_k) = \sigma_N(k)$. 
\end{mainlemma*}

Using the Main Lemma it is easy to see that Theorem~\ref{Thm:Main_thm} follows. Indeed, assuming it holds, we find the following:

\begin{proof}[Proof of Theorem~\ref{Thm:Main_thm}]
First, we have $|U_k| + |V_k| = 8k+4$. Let $n \in \mathbb{N}$ with $n \geq 4$, and let $k \in \mathbb{N}$ be the greatest $k$ such that $n \geq 8k+4$. In particular, $k = \lfloor \frac{n-4}{8}\rfloor$. Then by the Main Lemma we have
\begin{align*}
\partial_N(n) \geq \partial_N(8k+4) \geq \partial_N(U_k, V_k) &\sim N^{2k} = N^{2\lfloor \frac{n-4}{8}\rfloor} \sim N^{n/4}.
\end{align*}
Hence $\partial_N(n) \succeq N^{n/4}$, and $\Pi_N$ has at Dehn function $\partial_N(n)$ at least exponential. 
\end{proof}

The remainder of the article will be devoted to the proof of the Main Lemma.

\subsection{Proof of the Main Lemma}

Let us recall the definition of the words $U_k$ and $V_k$, for $k \geq 1$:
\[
U_k \equiv ab^{2k}a^{2k}a, \qquad V_k \equiv b^{2k-1}a^{2k}baa.
\]
We shall also very frequently make use of the abbreviation $X \equiv ba$. Thus, the defining relation of $\Pi_N$ is $XaX^N$. The overall strategy of the proof is to apply $\fA$ to the pair $(U_k, V_k)$. This process will be divided into two parts. The first will serve to reduce the pair $(U_k, V_k)$ to one of the form $(W, a)$, where the precise form of $W$ depends both on $k$ and $N$. The second part will be to reduce the word $(W, a)$ to a pair $(X^paX^{pN}, a)$, where $p = p(k, N)$ is a ``very large'' number. Of course, no matter the value of $p$, we have that Adian's algorithm reduces $(X^paX^{pN}, a)$ to $(\varepsilon, \varepsilon)$ in $X^p$ steps. This will complete our proof. We shall keep a precise count of how many steps are required. 

We begin with a somewhat technical lemma, but which is often very useful. In essence, the lemma says that $X$'s can be ``pushed'' to the right of an $aa$, as long as one applies an appropriate multiplicative factor of $N^2$. 

\begin{lemma}\label{Lem:1}
Let $p, q \geq 0$. Then $\fAA{pN+p}(X^paaX^q) = aaX^{q + pN^2}$. 
\end{lemma}
\begin{proof}
The proof of the claim is now by induction on $p$. If $p=0$, then there is nothing to prove. Assume the lemma is proved for all $p<p'$ with $p'>0$. We prove the claim for $p=p'$. As $p>0$, the prefix decomposition of $W \equiv X^paaX^q \equiv X^{p-1}baaaX^q$ is
\[
\fR(X^paaX^q) = X^{p-1} \mid baa \: \fbox{$a$} \: X^q,
\]
so we find that $\fAA{1}(W) = X^{p-1}baa(ba)^1aX^NX^q$. Here, the prefix decomposition is easily found to be 
\[
\fR(X^pabaaX^NX^q) = X^{p-1} \mid baa(ba)^1 \: \fbox{$a$} \: X^q,
\]
and so $\fAA{2}(W) = X^{p-1}baa(ba)^2aX^{2N}X^q$. Continuing iteratively, we thus find that 
\[
\fAA{N}(W) = X^{p-1}\underline{baa(ba)^N}aX^{N^2}X^q,
\]
where the head of the prefix decomposition of this word is the underlined segment $baa(ba)^N$. Consequently, 
\[
\fAA{N+1}(W) = X^{p-1}aaX^{q+N^2}
\]
and as $\fAA{(p-1)N+p-1}(X^{p-1}aaX^{q+N^2}) = aaX^{q+N^2+ (p-1)N^2}$ by the inductive hypothesis, we find that 
\[
\fAA{pN+p}(W) = \fAA{N+1+(p-1)N+p-1}(W) = aaX^{q+pN^2},
\]
which is what was to be proved. 
\end{proof}

We define three sequences $s_N, t_N$, and $T_N$ of natural numbers as follows:
\begin{align*}\label{Eq:defn-of-sn}
s_N(n) &= \begin{cases*} N, & if $n=0$, \\ N^2s_N(n-1) - N^2 + N, & if $n > 0$.\end{cases*} \\
t_N(n) &=  \begin{cases*} 0, & if $n=0$, \\ (N+1)(s_N(n-1)-1) + 1, & if $n > 0$.\end{cases*} \\
T_N(n) &= \sum_{i=0}^n t_N(i).
\end{align*}
Note that $t_N(n)$ is defined via a recurrence on $s_N$, rather than on $t_N$. For example, for $N=2$, the sequences begin as follows:
\begin{align*}
s_2(n) &\colon  2, 6, 22, 86, 342, 1366, \dots, \\
t_2(n) &\colon  0, 4, 16, 64, 256, 1024, \dots, \\
T_2(n) &\colon  0, 4, 20, 84, 340, 1364, \dots.
\end{align*}
The reader may well recognise the sequence $t_2(n)$ above, and that $s_2(n)$ appears very similar to $T_2(n)$; of course, this is no coincidence. Indeed, it is an easy exercise in undergraduate combinatorics to derive a closed form for the three sequences above, as follows:

\begin{lemma}\label{Lem:stT-closed-form}
The sequences $s_N, t_N$, and $T_N$ admit the following closed forms:
\begin{enumerate}
\item $s_N(n) = \frac{N}{N+1}(N^{2n+1}+1)$ for all $n \geq 0$. 
\item $t_N(n) = (N^2)^n$ for all $n > 0$. 
\item $T_N(n) = \frac{N^2}{N^2-1}(N^{2n}-1)$ for all $n \geq 0$. 
\end{enumerate}
\end{lemma}

Having diverged somewhat from considering words, we now return to our words $U_k$ and $V_k$. We begin with a useful lemma, which will use all three of the above functions. 

\begin{lemma}\label{Lem:2}
Let $q \geq 0$. Let $Q$ be an arbitrary word. Then
\[
\fAA{T_N(q)}(X^N(aX^N)^{2q}, a^{2q}Q) = (X^{s_N(q)}, Q).
\]
\end{lemma}
\begin{proof}
The proof is by induction on $q$. If $q=0$, then $X^N(aX^N)^{2q} = X^N$, and as $T_N(q) = 0$ and $s_N(q) = N$, the claim holds. Assume the claim holds for all $q < q'$ for some $q' >0$. We prove the claim for $q=q'$. 

As $X^N(aX^N)^{2(q-1)}$ is, by the inductive hypothesis, left divisible by $a$, it follows from Lemma~\ref{Lem:left-comp-lemma} and the inductive hypothesis that we have 
\[
\fAA{T_N(q-1)}(X^N(aX^N)^{2q}, a^{2q}Q) = (X^{s_N(q-1)}(aX^N)^2, a^2Q).
\]
We will now rewrite this pair further. First, we find the prefix decomposition of $X^{s_N(q-1)}(aX^N)^{2}$ as:
\begin{align*}
\fR(X^{s_N(q-1)}(aX^N)^{2}) &\equiv \fR(X^{s_N(q-1)-1}baaX^N(aX^N), \\&= X^{s_N(q-1)-1} \: \fbox{$baaX^N$} \: aX^N,
\end{align*}
and hence we find, replacing this head by $a$, that 
\begin{equation}\label{Eq:lem2-eq1}
\fAA{1}(X^{s_N(q-1)}(aX^N)^2, a^2Q) = (X^{s_N(q-1)-1}aaX^N), a^2 Q).
\end{equation}
By Lemma~\ref{Lem:1}, we have 
\begin{align}\label{Eq:lem2-eq2}
\fAA{(s_N(q-1)-1)(N+1)}(X^{s_N(q-1)-1}aaX^N) &= aaX^{N+(s_N(q-1)-1)N^2} \notag \\
&= aaX^{N^2s_N(q-1) - N^2 + N} \notag \\
&= aaX^{s_N(q)}.
\end{align}
Now, applying Lemma~\ref{Lem:left-comp-lemma} with $P \equiv X^{s_N(q-1)-1}aaX^N$, we can combine \eqref{Eq:lem2-eq1} and \eqref{Eq:lem2-eq2}, and thus find:
\begin{align}\label{Eq:lem2-eq3}
\fAA{(s_N(q-1)-1)(N+1) + 1}(X^{s_N(q-1)}(aX^N)^{2q}, a^{2q}Q)) &= (aaX^{s_N(q)}, aaQ)_\text{p-red} \notag \\
&= (X^{s_N(q)}, Q).
\end{align}
Thus we are almost done by induction; it remains to count the steps performed in transforming $(X^N(aX^N)^{2q}, a^{2q}Q)$ to $(X^{s_N(q)}, Q)$. By the inductive hypothesis, we first have $T_N(q-1)$ steps in transforming the pair $(X^N(aX^N)^{2q}, a^{2q}Q)$ to $(X^{s_N(q-1)}(aX^N)^2, a^2Q)$. The above reasoning shows that an additional 
\[
(s_N(q-1)-1)(N+1)+1
\] 
steps were subsequently performed. Thus the total number of steps performed was
\[
T_N(q-1) + \underbrace{(s_N(q-1)-1)(N+1)+1}_{t_N(q)} = T_N(q-1) + t_N(q) = T_N(q),
\]
which is what was to be shown. 
\end{proof}

\begin{lemma}\label{Lem:3}
Let $W$ be an arbitrary word. For all $k \geq 1$ and $N \geq 2$, we have 
\[
\fAA{T_N(k-1)+2k-1}(aW, V_k) = (X^{s_N(k-1)-1}W, a).
\]
\end{lemma}
\begin{proof}
Of course, the prefix decomposition of $aW$ has the distinguished letter $a$ as its head. In particular, we find $\fA(aW) = baaX^NW$, and hence, removing the shared prefix $b$ from $baaX^NW$ and $b$, we find:
\begin{align*}
\fAA{1}(aW, V_k) = (\underline{b} \cdot aaX^NW, \underline{b} \cdot b^{2k-2} a^{2k}baa)_\text{p-red} = (aaX^NW, b^{2k-2}a^{2k}baa).
\end{align*}
The head of $aaX^NW$ is the first distinguished $a$. Hence, continuing iteratively, we find that:
\begin{align*}
\fAA{2k-1}(aW, V_k) &= (\underline{baa} X^N\cdot (aX^N)^{2k-2}W, \underline{baa}a^{2k-2}baa)_\text{p-red} \\
&= (X^N(aX^N)^{2k-2}W, a^{2k-2}baa)_\text{p-red}.
\end{align*}
If $k=1$, then $(X^N(aX^N)^{2k-2}W, a^{2k-2}baa) \equiv (X^NW, baa)$, and as, by definition, $X^N \equiv (ba)^N$, we have that $(X^NW, baa)_\text{p-red} = (X^{N-1}W, a)$, as required. 

Thus, assume that $k > 1$. By Lemma~\ref{Lem:2}, we have 
\begin{align*}
\fAA{T_N(k-1)}(X^N(aX^N)^{2k-2}, a^{2k-2}baa) &= (X^{s_N(k-1)}, baa)_\text{p-red}, \\
&= (X^{s_N(k-1)-1}, a).
\end{align*}
As the proof of Lemma~\ref{Lem:2} shows that the prefix decomposition of each of the $T_N(k-1)$ intermediary words in this rewriting has a head, we may perform the above rewriting first when applying $\fA$ to the pair $(X^N(aX^N)^{2k-2}W, a^{2k-2}baa)$. Hence, combining this rewriting with the above $2k-1$ steps, we find
\begin{align*}
\fAA{T_N(k-1)+2k-1}(aW, V_k) &= (X^{s_N(k-1)-1}W, a).
\end{align*}
This is precisely what was to be shown. 
\end{proof}

We are now ready to enact the first part of Adian's algorithm on the pair $(U_k, V_k)$. It is worth keeping in mind that this part will only affect the first letter $a$ in the word $U_k$, but will transform the second word of the pair into $a$. 

\begin{lemma}\label{Lem:4}
For every $k \geq 1$ and $N \geq 2$, we have 
\[
\fAA{T_N(k-1)+2k-1}(U_k, V_k) = (X^{s_N(k-1)-1}b^{2k}a^{2k}a, a).
\]
\end{lemma}
\begin{proof}
This is immediate by Lemma~\ref{Lem:3}, taking $aW \equiv U_k \equiv ab^{2k}a^{2k}a$. 
\end{proof}

Thus we have completed the first part of Adian's algorithm when applied to the pair $(U_k, V_k)$. Already, this is a large number of steps -- note that $T_N(k)$ grows as $N^{2k}$, so if we could prove (by whatever means) that $U_k = V_k$ in $\Pi_N$, then we could conclude just on the basis of this first part that $\Pi_N$ has at least an exponential Dehn function. We continue with our precise count of the steps. 

The second part of Adian's algorithm now consists in proving that, for all $k \geq 1$ and $N \geq 2$, we have that $X^{s_N(k-1)-1}b^{2k}a^{2k}a$ is equal to $a$.

\begin{lemma}\label{Lem:5}
For every $p \geq 1$ and $q>1$ and $N \geq 2$, we have
\[
\fAA{pN+p+1}(b^{2q-1}X^p a^{2q}) = b^{2(q-1)-1} X^{pN^2-N} a^{2(q-1)}.
\]
\end{lemma}
\begin{proof}
Note that the word $b^{2q-1}X^pa^{2q}$ is of the form $b^{2q-1}X^{p}aaQ$ for a word $Q$. By Lemma~\ref{Lem:1}, we have that $\fAA{pN+p}(X^paa) = aaX^{pN^2}$, and as no word in the corresponding sequence of elementary transformations begins with $a$ except the last, i.e. $aaX^{pN^2}$, it hence follows from Lemma~\ref{Lem:double-comp-lemma} that 
\begin{align*}
\fAA{pN+p}(b^{2q-1}X^paa \cdot a^{2q-2}) = b^{2q-1}aaX^{pN^2}a^{2q-2}.
\end{align*}
Consequently, as the prefix decomposition of $b^{2q-1}aaX^{pN^2}a^{2q-2}$ is 
\[
\fR(b^{2q-1}aaX^{pN^2}a^{2q-2}) = b^{2q-2} \: \fbox{$baaX^{N}$} \: X^{pN^2-N}a^{2q-2},
\]
and hence 
\begin{align*}
\fAA{pN+p+1}(b^{2q-1}X^paa \cdot a^{2q-2}) &= b^{2q-2}aX^{pN^2-N}a^{2q-2} \notag \\
&\equiv b^{2(q-1)-1}X^{pN^2-N+1} a^{2(q-1)}.
\end{align*}
This is precisely what was to be shown.
\end{proof}

% S -> u -> s'
% T -> v -> t'
% T' -> w -> T'

For ease of bookkeeping, we define another three sequences $s'_N, t'_N$, and $T'_N$ of natural numbers as follows:
\begin{align*}\label{Eq:defn-of-sn}
s'_N(n) &= \begin{cases*} 1, & if $n=0$, \\ N^2s'_N(n-1)-N+1, & if $n > 0$.\end{cases*} \\
t'_N(n) &=  \begin{cases*} 0, & if $n=0$, \\ (N+1)s'_N(n-1) + 1, & if $n > 0$.\end{cases*} \\
T'_N(n) &= \sum_{i=0}^n t'_N(i).
\end{align*}
Note that $t'_N(n)$ is defined via a recurrence on $s'_N$, rather than on $t'_N$. For example, for $N=2$ these sequences begin:
\begin{align*}
s'_2(n) &\colon  1, 3, 11, 43, 171, 683, \dots, \\
t'_2(n) &\colon  0, 4, 10, 34, 130, 514, \dots, \\
T'_2(n) &\colon  0, 4, 14, 48, 178, 692, \dots.
\end{align*}

Analogously to Lemma~\ref{Lem:stT-closed-form}, we can easily find closed form expressions for these functions.

\begin{lemma}\label{Lem:uvw-closed-forms}
The sequences $s'_N, t'_N$, and $T'_N$ admit the following closed forms:
\begin{enumerate}
\item $s_N'(n) = \frac{N}{N + 1}(N^{2n} + \frac{1}{N})$ for all $n \geq 0$. 
\item $t_N'(n) = N^{2n}+2$ for all $n \geq 1$.
\item $T_N'(n) = \frac{N}{N^2 - 1}(N^{2n} - 1) + 2n$ for all $n \geq 0$. 
\end{enumerate}
\end{lemma}

We remark that an immediate consequence of Lemma~\ref{Lem:stT-closed-form} and Lemma~\ref{Lem:uvw-closed-forms} is the relationship
\begin{equation}\label{Eq:sn'-relationship-with-sn}
s'_N(k)-1 = N(s_N(k-1)-1).
\end{equation}
We shall presently use this. First, however, we will use our new functions $s_N', t_N'$, and $T_N'$ to count the number of steps performed by $\fA$ to prove that $b^{2k}a^{2k}a$ is left divisible by $a$. 

\iffalse
\comment{Different $S, T$ below}\\
\comment{$T_0 = 0$, so $T_1 = N+2$ and   $T_2 = (N+1)S_1 + 1$} \\
\comment{$S_0 = 1$, so $S_1 = N^2-N+1$ and $S_2 = N^2 S_1 -N+1$} \\ 
\comment{Note: $S(k) = Ns(k-1)-N+1$} \\
\comment{$T'(k) = \sum_{i=1}^{k-1} T_N(i)$ and has closed form $\frac{2}{3} (4^{k} -1)+2k$ for $N=2$.}
\comment{Well $w$ has closed form $sum_(n=1)^k(N^(2 n - 1) + 2) = (N (N^(2 k) - 1))/(N^2 - 1) + 2 k$, but this is sum from $1$ to $k$.}
\fi

\begin{lemma}\label{Lem:6}
For every $k \geq 1$ and $N \geq 2$, we have
\[
\fAA{T'_N(k)}(b^{2k}a^{2k}a) = aX^{s_N'(k)}.
\]
\end{lemma}
\begin{proof}
Consider the word $b^{2k}a^{2k}a \equiv b^{2k-1}X^1a^{2k}$. We may apply Lemma~\ref{Lem:5} to this word to find that 
\begin{align*}
\fAA{N+2}(b^{2k-1}X^1a^{2k}) &= b^{2(k-1)-1}X^{N^2-N+1}a^{2(k-1)} \\
\therefore \fAA{t_N'(1)}(b^{2k-1}X^1a^{2k}) &= b^{2(k-1)-1}X^{s_N'(1)}a^{2(k-1)}
\end{align*}
Of course, the word $b^{2(k-1)-1}X^{s_N'(1)}a^{2(k-1)}$ is also of the correct form for applying Lemma~\ref{Lem:5}, and we find 
\begin{align*}
\fAA{(N+1)s_N'(1)+1}(b^{2(k-1)-1}X^{s_N'(1)}a^{2(k-1)}) &= b^{2(k-2)-1}X^{N^2s_N'(1)-N+1}a^{2(k-2)}.
\end{align*}
As $t_N'(2) = (N+1)s_N'(1) + 1$ and $s_N'(2) = N^2s_N'(1)-N+1$, we conclude
\begin{align*}
\fAA{t_N'(1) + t_N'(2)}(b^{2k}a^{2k}a) &= b^{2(k-2)-1}X^{s_N'(2)}a^{2(k-2)}.
\end{align*}
This word is again of the required form to apply Lemma~\ref{Lem:5}, etc. and hence, continuing iteratively for $k-1$ steps and using $T'_N(k) = \sum_{i=1}^k t_N'(i)$, we find
\begin{align}\label{Eq:lem-6-1}
\fAA{T'_N(k-1)}(b^{2k}a^{2k}a) = bX^{s_N'(k-1)}a^2.
\end{align}
Rewriting $bX^{S_N(k-1)}a^2$ is easy. Indeed, by Lemma~\ref{Lem:1} we have 
\begin{align}\label{Eq:lem-6-2}
\fAA{(N+1)s_N'(k-1)}(bX^{s_N'(k-1)}a^2) = baaX^{N^2s_N'(k-1)}.
\end{align}
Note that $(N+1)s_N'(k-1) = t_N'(k)-1$. Furthermore, we obviously have 
\begin{align}\label{Eq:lem-6-3}
\fAA{1}(baaX^{N^2s_N'(k-1)}) = aX^{N^2s_N'(k-1)-N}, 
\end{align}
and $N^2s_N'(k-1)-N = s_N'(k)$. We can now assemble all our rewritings \eqref{Eq:lem-6-1}, \eqref{Eq:lem-6-2}, and \eqref{Eq:lem-6-3}, to find that 
\begin{equation*}\label{Eq:lem-6-4}
\fAA{T'_N(k-1)+t_N'(k) -1 + 1}(b^{2k}a^{2k}a) = aX^{s_N'(k)-1}
\end{equation*}
which is to say 
\[
\fAA{T'_N(k)}(b^{2k}a^{2k}a) = aX^{s_N'(k)-1}.
\]
This is precisely what was to be proved.
\end{proof}

Note that as no word, except the last, in the rewriting process described in the proof of Lemma~\ref{Lem:6} begins with $a$. In particular, we conclude by Lemma~\ref{Lem:double-comp-lemma} that Lemma~\ref{Lem:6} also describes the rewriting process for $X^q b^{2k}a^{2k}$, for any $q \geq 0$. In particular, taking $q=s_N(k-1)-1$, we conclude:

\begin{lemma}\label{Lem:7}
For every $k \geq 1$, we have 
\[
\fAA{T'_N(k)}(X^{s_N(k-1)-1}b^{2k}a^{2k}a, a) = (X^{s_N(k-1)-1}aX^{s_N'(k)-1}, a).
\]
\end{lemma}

By \eqref{Eq:sn'-relationship-with-sn}, the pair of words produced in Lemma~\ref{Lem:7} is of the form $(X^paX^{pN}, a)$. We are now almost done, as Adian's algorithm applied to any word $X^paX^{pN}$ will clearly result in $a$ after $p$ steps. We write out the exact statement and a brief proof. 

\begin{lemma}\label{Lem:8}
For every $k \geq 1$, we have 
\[
\fAA{s_N(k-1)-1}(X^{s_N(k-1)-1}aX^{s_N'(k)-1}, a) = (\varepsilon, \varepsilon).
\]
\end{lemma}
\begin{proof}
First, note that by \eqref{Eq:sn'-relationship-with-sn} we have that the word 
\[
X^{s_N(k-1)-1}aX^{s_N'(k)-1}
\]
is of the form $X^paX^{pN}$ for $p \geq 0$. For any word of this form, it is easy to see that it is equal to $a$; indeed, one may with little difficulty write down the sequence of transformations carried out by $\fA$ directly. One may also note that the rewriting system $XaX^N \to a$ will rewrite any word $X^paX^{pN}$ to $a$ in $p$ steps, and that the resulting sequence of elementary transformations is right directed -- indeed, there are no independent transformations. Hence, by Proposition~\ref{Prop:Adian-prop}, this sequence must be the one produced by $\fA$, so $\fAA{p}(X^paX^{pN}, a) = (\varepsilon, \varepsilon)$, as desired. 
\end{proof}

This completes the description of the second part of the action of $\fA$ when applied to the pair $(U_k, V_k)$. For brevity, let
\begin{equation}\label{Eq:sigma-def}
\sigma_N(k) = \underbrace{T_N(k-1)+(2k-1)}_{\text{Lemma~\ref{Lem:4}}} +\underbrace{T'_N(k)}_{\text{Lemma~\ref{Lem:7}}} + \underbrace{(s_N(k-1)-1)}_{\text{Lemma~\ref{Lem:8}}}.
\end{equation}

Then $\sigma_N(k)$ is the total number of steps taken by $\fA$ to rewrite the pair $(U_k, V_k)$ into $(\varepsilon, \varepsilon)$, as it is simply the sum of the number of steps in the indicated lemmas; explicitly, we have the following:

\begin{proposition}\label{Prop:Main-prop}
For every $N \geq 2$ and $k \geq 1$, we have 
\begin{align*}
\fAA{\sigma_N(k)}(U_k, V_k) = (\varepsilon, \varepsilon).
\end{align*}
Hence, the shortest sequence of elementary transformations verifying the identity $U_k = V_k$ in $\Pi_N$ has length $\sigma_N(k)$.
\end{proposition}
\begin{proof}
This follows by first performing the rewriting from Lemma~\ref{Lem:4}, followed by the one in Lemma~\ref{Lem:7}, followed finally by the one in Lemma~\ref{Lem:8}.
\end{proof}

We give a closed form expression for $\sigma_N(k)$.

\begin{lemma}\label{Lem:sigma-closed-form}
For every $N \geq 1$ and $k \geq 1$, we have 
\[
\sigma_N(k) = \frac{2N}{N^2-1}(N^{2k} - 1) + 4k -2.
\]
In particular, we have $\sigma_N(k) \sim N^{2k}$. 
\end{lemma}
\begin{proof}
This follows immediately from the definition \eqref{Eq:sigma-def}, the closed form expressions for $T_N(k-1)$ and $s_N(k-1)$ from Lemma~\ref{Lem:stT-closed-form}, the closed form expression for $T_N'(k)$ from Lemma~\ref{Lem:uvw-closed-forms}, and elementary algebraic manipulations.
\end{proof}

By Lemma~\ref{Lem:sigma-closed-form}, notice that for fixed $N$, we have 
\[
\sigma_N(k) = 2\left( 2k+1 +\sum_{i=1}^k N^{2k} \right).
\]
For example, for a fixed $N$ we have that $\sigma_N(k)$ for $k=1, 2, \dots$ begins:
\[
\sigma_N(k) \colon 2(N + 1), 2(N^3 + N + 3), 2(N^5 + N^3 + N + 5), \dots
\]
This is a particularly simple expression for $\sigma_N(k)$. As an example, when $N=2$, we have the following values for $k=1, 2, \dots$:
\[
\sigma_2(k)\colon 6, 26, 94, 354, 1382, 5482, 21870, \dots.
\]
In summary, for example taking $k=4$ and $N=2$, $\fA$ will terminate on the pair 
\[
(U_4, V_4) = (ab^8a^8a, b^7a^8baa),
\]
requiring a total of $354$ steps to do so. Similarly (and somewhat strikingly), if one instead takes $N=4$, then $\fA$ will take $34966$ steps to prove that $U_4 = V_4$ in $\Pi_4$.

The statement of the Main Lemma is now obtained by combining Proposition~\ref{Prop:Main-prop} and Lemma~\ref{Lem:sigma-closed-form}. This completes the proof of the Main Lemma. \qed

\iffalse
$k=3$. 
$T_2(2) = 20$.
$2k-1 = 5$ 
$T_N'(3) = 48$
$s_N(2) = 22$

$94$.

$k=5$.
$T_2(5-1) = 340$.
$2k-1 = 9$
$T_N'(4) = 692$.
$s_N(4) = 342$. 
$-1$ 
$1382$.

$k=6$.
$T_2(6-1) = 1364$.
$2k-1=11$.
$T_N'(5) = 2742$.
$s_N(5) = 1366$.
$-1$

\fi

\section{Some remarks on $\Pi_N$}\label{Sec:Some-remarks}

\noindent We make some general remarks on the monoids $\Pi_N$. First, we note that it is very tempting to conjecture that the Dehn function of $\Pi_N$ \textit{is} exponential, and not just bounded below by an exponential function. Proving this, however, seems somewhat intricate. Baumslag--Solitar groups $\BS(k, \ell)$ with $|k| \neq |\ell$ are well-known to have exponential Dehn function; this is proved in two steps. First, Gersten \cite[Theorem~B]{Gersten1992} gave a series of null-homotopic words whose van Kampen diagrams require a very large number of cells to fill. Thus the Dehn function is at least exponential. Second, one uses the fact that such Baumslag--Solitar groups are asynchronously automatic \cite[Corollary~E1]{Baumslag1991} and any group with this property has at most exponential Dehn function, see \cite{Epstein1992}. We have, in essence (using word transformations rather than diagrams), effected the first step for our ``Baumslag--Solitar''-esque monoids $\Pi_N$. However, the second seems more elusive, as automaticity is a less powerful property for semigroups than groups (see discussion below). We will, at least, prove that the Dehn function of $\Pi_N$ is always recursive. 

Note, first, that the \textit{group} defined by the same presentation as $\Pi_N$ is easily recognised as one of the famous characters of combinatorial group theory:
\begin{align*}
\pres{Gp}{a,b}{baa(ba)^N = a} &= \pres{Gp}{a,b}{baa(ba)^{N-1}b = 1} \\
&\cong \pres{Gp}{a,b}{bab^{N-1}ba^{-1}  =1} \\
&= \pres{Gp}{a,b}{a^{-1}ba  = b^{-N}} = \BS(1,-N),
\end{align*}
a metabelian Baumslag--Solitar group, where the indicated isomorphism comes from applying the free group automorphism induced by $a \mapsto a$ and $b \mapsto ba^{-1}$. All metabelian Baumslag--Solitar groups are residually finite, and hence have decidable word problem, providing an alternative route to using Magnus' theorem in this case. It may therefore be tempting to conjecture that $\Pi_N$ is also residually finite for every $N$. This, however, is not at all the case. In general, it is an open problem to classify which one-relation monoids are residually finite; however, in the monadic case $\pres{Mon}{a,b}{bUa = a}$, a full classification due to Bouwsma exists, but seems relatively unknown (e.g. it is not included in the reference list of \cite{NybergBrodda2021}). 

\begin{theorem*}[Bouwsma, 1993]
The monoid defined by $\pres{Mon}{a,b}{bUa = a}$ is residually finite if and only if $U \equiv b^k$ for some $k \geq 0$. 
\end{theorem*}

That is, the only residually finite $2$-generated monadic one-relation monoids are those of the form $\pres{Mon}{a,b}{b^ka = a}$. The result was never published, except in Bouwsma's 1993 Ph.D. thesis \cite[Corollary~2.5]{Bouwsma1993} (supervised by G. Lallement). From the above, we immediately conclude that $\Pi_N$ is \textit{not} residually finite for any $N$. We must therefore turn to an alternative approach to solving the word problem in $\Pi_N$. To do this, we turn to a remarkable result by Guba \cite{Guba1997} and the decidability of the rational subset membership problem in the metabelian Baumslag--Solitar groups $\BS(1, N)$. We refer the reader to the survey \cite{Lohrey2015} for details and definitions regarding the rational subset membership problem in groups.

\begin{theorem}\label{Thm:Pi_0-has-dec-wp}
For every $N \geq 2$, the word problem in $\Pi_N$ is decidable. In particular, the Dehn function of $\Pi_N$ is recursive. 
\end{theorem}
\begin{proof}
We follow Guba \cite[Theorem~2.1]{Guba1997}, and assume the reader is familiar with the details of this article. The submonoid of $\Pi_N$ generated by the suffixes of the defining word, i.e. the \textit{monoid of ends}, is denoted $S(\Pi_N)$. It is clear that this is generated by $a$ and $ba$. It follows immediately from the proof of \cite[Corollary~3.1]{Guba1997} that
\[
S(\Pi_N) \cong \pres{Mon}{x,y}{yxy^N = x}
\]
with the isomorphism $x \mapsto a$ and $y \mapsto ba$. This is a cycle-free presentation, so $S(\Pi_N)$ embeds in the one-relator group with the same presentation via the identity map, which is obviously $\BS(1,-N)$. In fact, to see this embedding is very simple, as every suffix of $baa(ba)^N$ is left divisible by $ba$. To see this, it suffices to show that $a$ is left divisible by $ba$, as every suffix of $baa(ba)^N$ begins either with $ba$ or with $a$. But $ba \cdot a(ba)^N = a$, so obviously this holds. Hence, the graph of ends of $\Pi_N$ has only a single component (cf. \cite[p. 1145]{Guba1997}), and hence $\overline{G}(\Pi_N) = \BS(1,-N)$. 

The word problem for $\Pi_N$ reduces to its left divisibility problem. By Oganesian's result \cite[Theorem~1]{Oganesian1982}, the left divisibility problem in $\Pi_N$ reduces to the left divisibility problem in $S(\Pi_N)$. As $S(\Pi_N)$ embeds in $\overline{G}(\Pi_N)$ via the identity map, by the above, the left divisibility problem in $S(\Pi_N)$ is hence easily reduced to the submonoid membership problem (indeed, the suffix membership problem) in $\overline{G}(\Pi_N) = \BS(1, -N)$, see \cite[Corollary~2.1]{Guba1997} for details. But the submonoid membership problem, indeed the rational subset membership problem, is decidable for any $\BS(1,N)$ with $N \geq 2$ by \cite[Theorem~3.3]{Cadilhac2020}. It is well-known that $\BS(1, N)$ and $\BS(1, -N)$ are commensurable (see e.g. \cite[Lemma~6.1]{CasalsRuiz2021}), and as decidability of the rational subset membership problem is inherited by taking subgroups (obviously) and by taking finite extensions (by Grunschlag \cite{Grunschlag1999}), we conclude that the rational subset membership problem, and hence also the submonoid membership problem, is decidable in $\BS(1, -N)$. This solves the word problem in $\Pi_N$. 
\end{proof}

\begin{remark}
The non-Russian-reading reader interested in Oganesian's remarkable work on the connection between left cycle-free monoids $\Pi$ and their associated semigroup of ends $S(\Pi)$ may first be inclined to consult the English translation of \cite{Oganesian1982}. Unfortunately, this translation is rather poor, and can make for confusing reading.\footnote{For two brief examples, on p. 89 (of the translation), the system (2) is said to have ``no cycles'' instead of the correct ``no left cycles''; on p. 92, one reads ``it is \textit{impossible} to find out whether $X$ is divisible by $d$'', but it should read ``it is \textit{possible}''. Further examples are not hard to find.} Instead, we advise the reader to first consult Guba's overview of Oganesian's work in the case of a single relation, found in \cite{Guba1997}.
\end{remark}

\begin{remark}
One can prove, using automata-theoretic methods, that $\fA$ in fact always terminates for $\Pi_N$ whenever $N \geq 1$, see \cite[Theorem~4.3]{Bouwsma1993}. This provides an alternative proof of Theorem~\ref{Thm:Pi_0-has-dec-wp}, but as the PhD thesis \cite{Bouwsma1993} is not easily accessible, we have opted for the above proof. Note that as $\fA$ always terminates, it follows, using an encoding due to Guba (see \cite[\S6.4]{NybergBrodda2021} for details) that the Collatz-like function $f_N \colon \mathbb{N} \times \mathbb{N} \to \mathbb{N}$ defined by:
\[
f_N(m, n) 
\begin{cases}
\left(\lfloor \frac{m}{2}\rfloor, \lfloor \frac{n}{2}\rfloor\right) & \text{if } m \equiv n \mod 2,\\
\left(\frac{m}{2}, 2^{2N+1}n + \frac{1}{3}(2^{2N+1}+1)\right) & \text{if } m \not\equiv n \text{ and } m \equiv 0 \mod 2, \\
(n, m) & \text{if } m \not\equiv n \text{ and } m \equiv 1 \mod 2.\\
\end{cases}
\]
always terminates for any input pair $(m, n) \in \mathbb{N} \times \mathbb{N}$ and any $N \geq 2$, where termination is meant in the sense that the sequence 
\[
(m, n) \to f_N(m,n) \to f_N^2(m, n) \to \dots 
\]
eventually results in $f_N^\ell(m, n) = (0, k)$ or $(k, 0)$ for some $k,\ell \in \mathbb{N}$.
\end{remark}

\subsection{Superexponential Dehn functions}

It would be interesting to see how quickly the Dehn function of a one-relator monoid can grow, even if exhibiting a non-recursive growth currently seems somewhat out of reach. A good starting point seems to be the following:

\begin{question}\label{Quest:monadic-with-tower?}
Does there exist a one-relation monoid $\pres{Mon}{a,b}{bUa=a}$ whose Dehn function is not bounded above by any finite tower of exponentials?
\end{question}

The exponential growth found as a lower bound for the Dehn function of $\Pi_N$ in this article is, of course, connected to the exponential Dehn function of the Baumslag--Solitar group $\BS(1, -N)$. One way of resolving Question~\ref{Quest:monadic-with-tower?} would be connected with the following well-known result. Gersten \cite{Gersten1992} proved that the one-relator group 
\[
\operatorname{BG} = \pres{Gp}{a,b}{[a, bab^{-1}] = a^{-1}}
\]
introduced by Baumslag \cite{Baumslag1969}, has a Dehn function which is not bounded above by any finite tower of exponentials. This group, however, is not a positive one-relator group, as it is easily observed to not be residually solvable (indeed, it is clear that $a$ lies in every term of the derived series), but every positive one-relator group is residually solvable \cite{Baumslag1971}. We therefore pose the following, purely group-theoretic, question. 

\begin{question}\label{Quest:pos-with-tower?}
Does there exist a positive one-relator group whose Dehn function is not bounded above by any finite tower of exponentials?
\end{question}

While a positive answer to Question~\eqref{Quest:pos-with-tower?} would not have a direct implication for Question~\ref{Quest:monadic-with-tower?} (or vice versa), it would seem to have a strong indication towards a positive answer.

\subsection{Finite complete rewriting systems and automaticity}

In \cite{Cain2013b}, Cain \& Maltcev claimed that every one-relation monoid with a presentation of the form 
\[
\pres{Mon}{a,b}{b^\alpha a^\beta b^\gamma a^\delta b^\varepsilon a^\varphi = a},
\]
i.e. a monadic one-relation monoid with defining relation $bUa = a$ in which $bUa$ has relative length $6$, admits a finite complete rewriting system. This claim is repeated (by the author of the present article) in \cite[p. 339]{NybergBrodda2021}. Cain \& Maltcev further claim that using these rewriting systems, one may observe that every such monoid has at most quadratic Dehn function. By the results in this present article, this cannot, however, be correct; taking $N=2$, the monoid $\Pi_2$ is defined by 
\[
\Pi_2 = \pres{Mon}{a,b}{baababa = a},
\] 
but by Theorem~\ref{Thm:Main_thm} its Dehn function is at least exponential. The source of the error is that their claimed finite complete rewriting system for $\Pi_2$ is not complete. This leads us to pose the following question. 

\begin{question}\label{Quest:FCRS?}
Does $\Pi_N$ admit a finite complete rewriting system for all $N \geq 2$?
\end{question}

Of course, an affirmative answer to Question~\ref{Quest:FCRS?} would give an alternative proof of decidability of the word problem in $\Pi_N$ (i.e. Theorem~\ref{Thm:Pi_0-has-dec-wp}). Finally, we note a connection with automaticity and the fellow traveller property, and ask the following natural question (we refer the reader to \cite{Hoffmann2006} for definitions are more details). 

\begin{question}\label{Quest:are-they-automatic?}
Are the monoids $\Pi_N$ (bi)automatic?
\end{question}

By \cite[Theorem~4.1]{Wang2009}, any left cancellative automatic monoid whose automatic structure satisfies the fellow traveller property has an at most quadratic Dehn function. An affirmative answer to Question~\ref{Quest:are-they-automatic?} would thus, in view of Theorem~\ref{Thm:Main_thm}, give a rather simple family of left cancellative automatic monoids without the fellow traveller property. 

\begin{remark}
We wish to make a final remark regarding a well-known result due to Ivanov, Margolis \& Meakin \cite{Ivanov2001}. This asserts (as a special case) that the word problem for any monoid $\pres{Mon}{a,b}{bUa = a}$ reduces to the word problem for the special one-relation inverse monoid $\pres{Inv}{a, b}{a^{-1}bUa =1}$. In particular, the word problem for $\Pi_N$ reduces to the word problem for $I_N = \pres{Inv}{a,b}{a^{-1}baa(ba)^N = 1}$. Based on the results above, we conjecture that the Dehn function for $I_N$ is also (at least) exponential, and indeed that the word problem for $I_N$ is decidable. 
\end{remark}

\section*{Acknowledgments}

I wish to express my sincere gratitude to Janet Hughes (Penn. State Univ.) for her reference services, and to Alan Cain (Univ. Nova de Lisboa) for useful discussions pertaining to \cite{Cain2013, Cain2013b}.

{
\bibliography{ExpDehnOneRelation.bib} 
\bibliographystyle{plain}
}
 \end{document}